\documentclass[envcountsame,envcountsect,runningheads,a4paper]{llncs}
\usepackage[latin1]{inputenc} 
\usepackage[ngerman, english]{babel}
\usepackage{tikz}
\usepackage{mathtools}
\usepackage{amsmath}
\usepackage{amssymb}
\usepackage{amsfonts}
\usepackage{amsxtra}
\usepackage{stmaryrd}
\usepackage{amssymb}
\usepackage{graphicx}

\usepackage{xspace}
\usepackage{latexsym}

\usepackage{url}
\urldef{\mailsa}\path|{alfred.hofmann, ursula.barth, ingrid.haas, frank.holzwarth,|
\urldef{\mailsb}\path|anna.kramer, leonie.kunz, christine.reiss, nicole.sator,|
\urldef{\mailsc}\path|erika.siebert-cole, peter.strasser, lncs}@springer.com|    
\newcommand{\keywords}[1]{\par\addvspace\baselineskip
\noindent\keywordname\enspace\ignorespaces#1}

\DeclarePairedDelimiter\ceil{\lceil}{\rceil}
\DeclarePairedDelimiter\floor{\lfloor}{\rfloor}

\newcommand{\m}[4]{\left( \begin{array}{ccc} #1 & \, & #2 \\ #3 & \, & #4 \end{array} \right)}

\newcommand{\mm}[9]{\left( \begin{array}{ccccc} 
#1 & \, & #2 \, & #3\\ 
#4 & \, & #5 \, & #6\\
#7 & \, & #8 \, & #9
\end{array} \right)}

\newcounter{commentcounter}

\begin{document}

\mainmatter  
\title{Pairwise Well-Formed Modes and Transformations}

\titlerunning{Pairwise Well-Formed Modes and Transformations}

\author{David Clampitt and Thomas Noll}
\authorrunning{David Clampitt and Thomas Noll}
\institute{Ohio State University, School of Music, USA \\
clampitt.4@osu.edu \vspace{9 pt} \\
Escola Superior de M\'{u}sica de Catalunya \\
Departament de Teoria, Composici\'{o} i Direcci\'{o}, Spain \\ thomas.mamuth@gmail.com}

\toctitle{Pairwise Well-Formed Modes and Transformations}

\maketitle

\begin{abstract}
One of the most significant attitudinal shifts in the history of music occurred in the Renaissance, when an emerging triadic consciousness moved musicians towards a new scalar formation that placed major thirds on a par with perfect fifths. In this paper we revisit the confrontation between the two idealized scalar and modal conceptions, that of the ancient and medieval world and that of the early modern world, associated especially with Zarlino. We do this at an abstract level, in the language of algebraic combinatorics on words. In scale theory the juxtaposition is between well-formed and pairwise well-formed scales and modes, expressed in terms of Christoffel words or standard words and their conjugates, and the special Sturmian morphisms that generate them. Pairwise well-formed scales are encoded by words over a three-letter alphabet, and in our generalization we introduce special positive automorphisms of $F3$, the free group over three letters.                 


\end{abstract}

\section{Introduction:  Authentic and Triadic  Modes}

Figure \ref{fig:Major} shows a C-major scale with two  different interpretations of its step interval pattern. In the annotation $aaba|aab$ (above the staff) the two letters $a$ and $b$ designate the major and minor steps, respectively. The vertical stroke $|$ designates the authentic divider of the mode into a species of the fifth $aaba$ and a species of a fourth $aab$. This pattern is called the \emph{Authentic division of the Ionian Mode}. In the annotation $ac|ba||cab$ (below the staff) the three letters $a, c$ and $b$ designate the greater and lesser major and the minor steps, respectively. Together they divide the major mode triadically into a species of the major third $ac$, a species of the minor third $ba$ and a species of the fourth $cab$. This pattern shall be called the \emph{Triadic Division of the Ionian Major Mode}.  
   
\begin{figure}
\begin{center}
\begin{minipage}{80 mm}
\resizebox*{8 cm}{!}
{\includegraphics{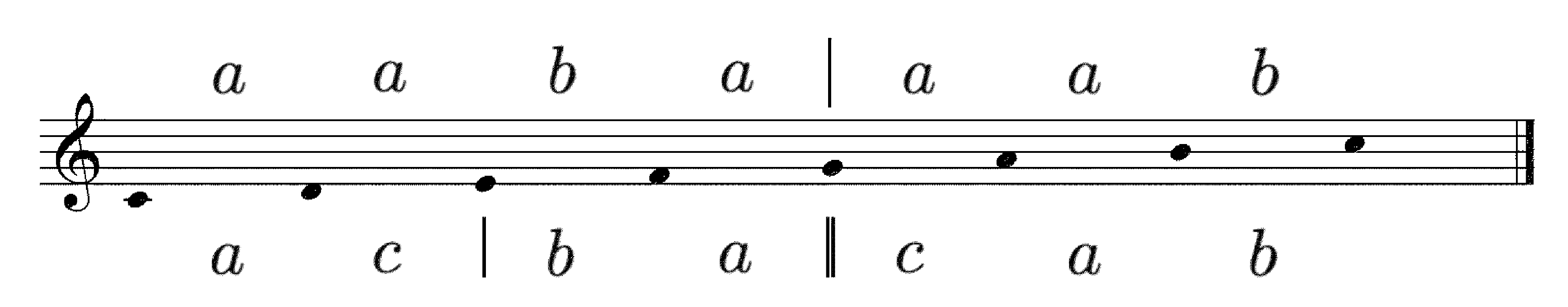}}%
\caption{Authenic Division of the Ionian and Triadic Division of the Ionian Major mode.}%
\label{fig:Major}
\end{minipage}
\end{center}
\end{figure}

The contrasting juxtaposition evokes several open questions of historical and systematic nature about the particular relevance of these modes for different types of music and music analysis. Within the discourse of mathematical music theory they point to the theory of well-formed scales and modes (\cite{CareyClampitt89}, \cite{CareyClampitt96a}, \cite{Noll2009}, \cite{ClampittNoll2011}), on the one hand, and to the theory of pairwise well-formed scales (\cite{Clampitt1997}, \cite{Clampitt2007}), on the other. In the present article we therefore extend some transformational innovations within the theory of well-formed modes in order to make them fruitful within a theory of  pairwise well-formed modes. These investigations eventually contribute to a deeper theoretical understanding of the juxtaposition in Fig. \ref{fig:Major}.

\section{Non-Singular Pairwise Well-Formed Modes}\label{PWWF}
In this section we revisit some important results from \cite{Clampitt1997} about the structure of non-singular pairwise well-formed scales and re-interpret them in a word-theoretic context.
The $3$-letter word $acbacab$ describes the species of the octave of the Ionian Major mode, and thereby it is the step interval pattern of a \emph{pairwise well-formed scale}. The motivation behind this concept is the following: The two-letter word $aabaaab$, describing the Ionian species of the octave can be obtained from $acbacab$ by an identification of the letter $c$ with the letter $a$: $\pi_{c \to a}(acbacab) = aabaaab$. In traditional music-theoretical terms this letter projection describes \emph{syntonic identification}, i.e., neglecting the difference between the greater and lesser major steps. There are two more such letter identifications, both of which lead to well-formed modes,  whence the term \emph{pairwise well-formed}. One of them is $\pi_{b \to c}(bacabac) = cacacac$. It describes an identification of the minor step with the lesser major step, i.e., neglecting their difference, what we may call the (harmonic) \emph{apotome}. This mode neutralizes also the difference between the major and minor thirds and can be seen as a modal refinement of the generic third-generated scale. The third letter projection $\pi_{a \to b}(bacabac) = bbcbbbc$ identifies the minor step with the greater major step. The neglected interval is the sum of the two previously mentioned ones, and so we can formally speak of an \emph{apo-syntonic identification}. It is arguable, though, whether this third projection bears a direct musical meaning. As an auxiliary construction it proves to be very useful on a theoretical level. This becomes clear in the course of the article. 

We will represent non-singular pairwise well-formed (PWWF) scales by words over a three-letter alphabet $A = \{a, b, c\}$, as in \cite{Clampitt2007}, and we will henceforth refer to PWWF words and drop the qualifier ``non-singular'' (the singular case is represented by the word $abacaba$ and its word-theoretical conjugates).  We denote the set of all PWWF words by $\mathfrak{A} \subset \{a, b, c\}^{\ast}$. From \cite{Clampitt1997}, \cite{ClampittNoll2011}, $v \in \mathfrak{A}$ if and only if each of the three projections $\pi_{x \to y}$ (described above) results in a well-formed word, i.e., the conjugate of a Christoffel word, equivalently, of a standard word (see \cite{Lothaire}, \cite{Berthe} for the relevant word theory background). For every word $v \in \mathfrak{A}$,  the length $|v|$ is odd, and the multiplicities of two of the letters are the same: $|v|_b = |v|_c$. It follows that $|v|_a$ is odd. In light of these facts, given a special standard word, we can always construct a PWWF word, by the \emph{bisecting substitution} defined below.

\begin{definition}
Consider a special standard morphism $f$ acting on the word monoid $\{a, c\}^\ast$ and consider the word $w = f(ac) = w_1w_2 \dots w_n$. We further suppose that $|w| = n$ is odd  and that $|w|_c$ is even. Then we define the \emph{bisection} of the word $w$ as $w_\prec = v = v_1v_2 \dots v_n \in \{a, b, c\}^\ast$ with $$v_k := \left \{ 
\begin{array}{lll} 
a & \mbox{if} & w_k = a, \\                                                                                                                                                                                                                                                                         
b & \mbox{if} & w_k = c \mbox{ and } |w_1 \dots w_k|_c \mbox{ is odd} \\
c & \mbox{if} & w_k = c \mbox{ and }  |w_1 \dots w_k|_c  \mbox{ is even.} 
\end{array} \right.$$ The bisecting substitution $\sigma: \{a, b, c\}^\ast \to \{a, b, c\}^\ast$ is then defined as $$\sigma(a) = v_1 \dots v_m, \quad \sigma(b) = v_{m+1} \dots v_{2m},  \quad \sigma(c) = v_{2m+1} \dots v_{n}, \, \mbox{where  } m = |f(a)|.$$
\end{definition}

{\bf Remark}: PWWF scales have distinct inversions, whereas the inversion of a mode of a well-formed scale is a mode of that scale (e.g., Ionian inverted is Phrygian). In word-theoretical language, if $w$ is a standard word, the reversal of $w$ is in the conjugacy class of $w$. For $w \in \mathfrak{A}$, the reversal of $w$ is in its own conjugacy class, distinct from that of $w$. For $k \in \mathbb{N}$, there are $\phi(k)/2$ distinct conjugacy classes of standard words of length $k$; for PWWF words ($k$ odd), there are $\phi(k)$ distinct conjugacy classes \cite{Clampitt1997}. The defining projections are insensitive to reversal, however, up to trivial replacements of the letters, so we may pair $w$ with its reversal, and choose whichever is convenient as representive. There is thus a bijection between conjugacy classes of standard words and classes of PWWF words of odd length $k$.

For example, $w = bacabac \in \mathfrak{A}$, and its reversal is $w' = cabacab$. $\pi_{c \to a}(w) = baaabaa$ (representing Phrygian), $\pi_{b \to a}(w) = aacaaac$ (representing Ionian), and $\pi_{b \to c}(w) = cacacac$ (representing, e.g., Dorian thirds); while $\pi_{c \to a}(w') = aabaaab$ (Ionian), $\pi_{b \to a}(w') = caaacaa$ (Phrygian), and $\pi_{b \to c}(w) = cacacac$. Therefore, we may depart from either PWWF representative. 
 
\begin{proposition}\label{Msigma} Consider a PWWF  substitution $\sigma$ and let $f(a) = \pi_{b \to c}(\sigma(a))$, $f(c) = \pi_{b \to c}(\sigma(bc))$ denote its apotomic projection. Let $M_f = \m {|f(a)|_a} {|f(c)|_a} {|f(a)|_c}  {|f(c)|_c}$ denote the incidence matrix of $f$. Then the incidence matrix of $\sigma$ is given as 
$$\tilde{M}_\sigma =  \mm  {|f(a)|_a} {|f(a)|_a} {|f(c)|_a - |f(a)|_a} 
{\floor{\displaystyle  \frac{|f(a)|_c}{2}}} {\ceil{\displaystyle \frac{|f(a)|_c}{2}}}                                                                                                                                                                                                                                                                                                                                                                                                                                                                                                                                                                                                    {\displaystyle \frac{|f(c)|_c - |f(a)|_c}{2}} 
{\ceil{\displaystyle  \frac{|f(a)|_c}{2}}} {\floor{\displaystyle \frac{|f(a)|_c}{2}}} {\displaystyle \frac{|f(c)|_c - |f(a)|_c}{2}} 
.$$ \end{proposition}

\section{Apo-syntonic Conversion}
In addition to the letter projections $\pi_{b \to c}, \pi_{a \to b}, \pi_{c \to a}   $ on words $v \in A^{\ast}$, we apply them to substitutions as follows (using the same symbols):
\begin{definition}\label{PWWFsubstitution} Consider a substitution (a monoid morphism) $\sigma: A^{\ast} \to A^{\ast}$. Then we obtain three induced substitutions $\pi_{b \to c}(\sigma): \{a, c\}^\ast \to \{a, c\}^\ast$, $\pi_{a \to b}(\sigma): \{b, c\}^\ast \to \{b, c\}^\ast$ and $\pi_{c \to a}(\sigma): \{a, b\}^\ast \to \{a, b\}^\ast$ by virtue of:
$$\begin{array}{llllllllll}
& [\pi_{b \to c}(\sigma)](a) & := & \pi_{b \to c}(\sigma (a)) & \quad & [\pi_{b \to c}(\sigma)](c) & := & \pi_{b \to c}(\sigma (bc)) \\
& [\pi_{a \to b}(\sigma)](b) & := & \pi_{a \to b}(\sigma (ab)) & \quad & [\pi_{a \to b}(\sigma)](c) & := & \pi_{a \to b}(\sigma (c))\\
& [\pi_{c \to a}(\sigma)](a) & := & \pi_{c \to a}(\sigma (ab)) & \quad & [\pi_{c \to a}(\sigma)](b) & := & \pi_{c \to a}(\sigma (c)) \\
\end{array}$$
We say that $\sigma$ is an authentic PWWF substitution iff all three projections $\pi_{b \to c}(\sigma)$, $\pi_{a \to b}(\sigma)$ and $\pi_{c \to a}(\sigma)$ are Special Sturmian morphisms.
\end{definition}  
In the rest of this section we will assume that $\sigma$ is the bisecting substitution associated with a suitable special standard morphism $f$ and hence $f = \pi_{b \to c}(\sigma)$. Further we use the symbols $g$ and $\tilde{g}$ for the projections $g = \pi_{a \to b}(\sigma)$ and $\tilde{g} = \pi_{c \to a}(\sigma)$ The diagram in Figure \ref{fig:bisection} shows the interplay of $\sigma$ with its three projections.       
    
\begin{figure}
\begin{center}
\unitlength 1.1 cm
\begin{picture}(5, 2)(-2.5, -0.2)
\put(0.1,2){\makebox(0,0){$\sigma$}}
\put(1,-0.1){\makebox(0,0){$g$}}
\put(-3,-0.1){\makebox(0,0){$f$}}
\put(2.95,-0.1){\makebox(0,0){$\tilde{g}$}}
\put(0,1.8){\vector(-3, -2){2.7}}
\put(-2.8, 0.1){\vector(3, 2){2.7}}
\put(0.05, 1.8) {\vector(1, -2){0.85}}
\put(0.1, 1.8) {\vector(3, -2){2.7}}
\put(-2.75,-0.1){\vector(1,0){3.5}}
\put(1.2,- 0.1){\vector(1,0){1.6}}
\put(-2.2,1.2){\makebox(0,0){\small{bisection}}}
\put(-1,-0.3){\makebox(0,0){\small{apo-syntonic conversion}}}
\put(2.0,-0.3){\makebox(0,0){\small{conjugation}}}
\put(-1.0,0.8){\makebox(0,0){$\pi_{b \to c}$}}
\put(1,0.8){\makebox(0,0){$\pi_{a \to b}$}}
\put(2.0,0.8){\makebox(0,0){$\pi_{c \to a}$}}
\end{picture}
\caption{Interplay of a PWWF substitution $\sigma$ with its projections $f$, $g$ and $\tilde{g}$.}%
\label{fig:bisection}
\end{center}
\end{figure}
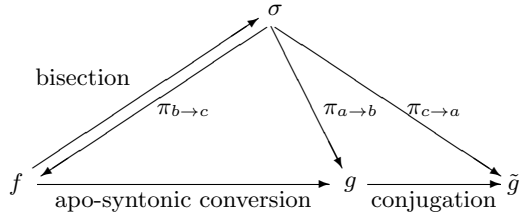

Our goal is now to understand the interdependence between $f$ and $g$. 

\begin{proposition}\label{Mg}
Consider an authentic (non-singular) PWWF mode $\sigma \in \mathfrak{A}$ with the projections $f = \pi_{b \to c} (\sigma), g = \pi_{a \to b} (\sigma)$ and $\tilde{g} = \pi_{c \to a} (\sigma)$. Then the common incidence matrix $M_g = M_{\tilde{g}}$  of $g$ and $\tilde{g}$ can be expressed in terms of the coefficients of the incidence matrix $M_f = \m {|f(a)|_a} {|f(c)|_a} {|f(a)|_c}  {|f(c)|_c}$ as follows:                                                                            
$$M_g = M_{\tilde{g}} = \m {2 |f(a)|_a + |f(a)|_c} {|f(c)|_a - |f(a)|_a + \frac{1}{2}(|f(c)|_c - |f(a)|_c)} {|f(a)|_c} {\frac{1}{2}(|f(c)|_c - |f(a)|c_b)}.$$
\end{proposition}
\begin{proof} The incidence matrix $M_g$ can be obtained from adding the first two columns and the first two rows of $\tilde{M}_\sigma$. Thus, after proposition \ref{Msigma} the upper left entry of $M_g$ becomes $ {|f(a)|_a} + {|f(a)|_a}  + {\floor{\displaystyle  \frac{|f(a)|_c}{2}}}  + {\ceil{\displaystyle \frac{|f(a)|_c}{2}}} = 2 |f(a)|_a +  f(a)|_c$. The upper right entry becomes ${|f(c)|_a - |f(a)|_a} +  {\displaystyle \frac{|f(c)|_c - |f(a)|_c}{2}}$, the lower left entry becomes ${\ceil{\displaystyle  \frac{|f(a)|_c}{2}}}+ {\floor{\displaystyle \frac{|f(a)|_c}{2}}} = |f(a)|_c$. The lower right entry remains $ {\displaystyle \frac{|f(c)|_c - |f(a)|_c}{2}}$.
\end{proof}

In order to understand the connection between $f$ and $g$ more directly, rather than via the substitution $\sigma$, we have a closer look at the structure of the linear map: $\tilde{\beta}: GL_2(\mathbb{R}) \to GL_2(\mathbb{R})$ with $$\tilde{\beta}(\m {a_{11}} {a_{12}} {a_{21}} {a_{22}}) := \m {2 a_{11} + a_{21}} {a_{12} - a_{11} + (a_{22} - a_{21})\slash 2} {a_{21}} {(a_{22} - a_{21})\slash 2}.$$
With $R = \m 1 1 0 1$ let $\Sigma^2$ and $\Sigma_2$ denote the following two subsets of $SL_2(\mathbb{N})$:
$$\begin{array}{lll} 
\Sigma^2 & = & \left \{ R \cdot \m {2 b_{11}} {b_{12}} {b_{21}} {b_{22}} \quad | \, {b_{11}}, {b_{12}}, {b_{21}}, {b_{22}} \in \mathbb{N}, 2 {b_{11}} {b_{22}} - {b_{12}} {b_{21}} =1 \right \}  \\
\Sigma_2 & = & \left \{ \m {b_{11}} {b_{12}} {b_{21}} {2 b_{22}} \cdot R   \quad | \, {b_{11}}, {b_{12}}, {b_{21}}, {b_{22}} \in \mathbb{N}, 2 {b_{11}} {b_{22}} - {b_{12}} {b_{21}} =1 \right \}  \\
\end{array}$$ 

\begin{lemma} $SL_2(\mathbb{N}) \cap \tilde{\beta}^{-1}(SL_2(\mathbb{N})) = \Sigma_2$ and $\tilde{\beta}(\Sigma_2) = \Sigma^2$.
\end{lemma}
\begin{proof}
Consider an arbitrary matrix $X = \m {c_{11}} {c_{12}} {c_{21}} {c_{22}} \in SL_2(\mathbb{N})$. Then we have $$\tilde{\beta}^{-1}(\m {c_{11}} {c_{12}} {c_{21}} {c_{22}}) = 
\m {\frac{c_{11}-c_{21}}{2}} {\frac{c_{11}-c_{21}}{2} + c_{12} - c_{22}} {c_{21}} {c_{21} + 2 c_{22}}.$$ The entry $c_{21} + 2c_{22}$ is larger than $c_{21}$ and therefore $\tilde{\beta}^{-1}(X) \in  SL_2(\mathbb{N})$ iff $Y = \tilde{\beta}^{-1}(X) \cdot R^{-1} = \m {\frac{c_{11}-c_{21}}{2}} {c_{12} - c_{22}} {c_{21}} {2 c_{22}} \in  SL_2(\mathbb{N})$. But in order to have a positive entry $c_{12} - c_{22}$ it turns out that $X \in SL_2(\mathbb{N})$ cannot be arbitrary. Also $Z = R^{-1} \cdot X = \m {c_{11} - c_{21}} {c_{12} - c_{22}} {c_{21}} {c_{22}}$  must be in $SL_2(\mathbb{N})$. And this implies $c_{11} \ge c_{21}$. Finally, we have to deal with the condition that the upper left entry $\frac{c_{11} - c_{21}}{2}$ of $\tilde{\beta}^{-1}(X)$ needs to be an integer. This implies that $c_{11} - c_{21}$, which is also the upper left entry of $Z$, is even, i.e., $RZ = X \in \Sigma^2$.
\end{proof}

\begin{corollary} The set $\Sigma_2$ parametrizes the conjugation classes of all authentic PWWF substitutions $\sigma$ in terms of the incidence matrices $M_f$ of their associated apotomic projections: the special standard morphisms $f = \pi_{b \to c}(\sigma)$. Also the set $\Sigma^2$ parametrizes these same conjugation classes by virtue of the incidence matrices $M_g$ of their associated apo-syntonic projections $g= \pi_{a \to b}(\sigma)$.
\end{corollary}

This motivates the following definition:
\begin{definition} The restriction of $\tilde{\beta}$ to the subset $\Sigma_2$ is called the \emph{apo-syntonic conversion}: $$\beta: \Sigma_2 \to \Sigma^2.$$
\end{definition}

Let $\delta: SL_2(\mathbb{N}) \to SL_2(\mathbb{N})$ denote the main-diagonal-flip $$\delta(\m {b_{11}} {b_{12}} {b_{21}} {b_{22}}) := \m {b_{22}} {b_{12}} {b_{21}} {b_{11}}.$$

\begin{lemma} $\delta$ mutually exchanges $\Sigma^2$ and $\Sigma_2$, i.e. $\delta(\Sigma^2) = \Sigma_2$ and $\delta(\Sigma_2) = \Sigma^2$.  
\end{lemma}

\begin{proposition}\label{commdiag} We have the following commutative diagram

\begin{center}
\unitlength 1 cm
\begin{picture}(4,2.5)
\put(0,2){\makebox(0,0){$\Sigma_2$}}
\put(3,2){\makebox(0,0){$\Sigma^2$}}
\put(0,0){\makebox(0,0){$\Sigma^2$}}
\put(3,0){\makebox(0,0){$\Sigma_2$}}
\put(0.7,2){\vector(1,0){1.5}}
\put(2.2,0){\vector(-1,0){1.5}}
\put(-0.1,1.7){\vector(0,-1){1.4}}
\put(3.0,1.7){\vector(0,-1){1.4}}
\put(0.05,0.3){\vector(0,1){1.4}}
\put(3.15,0.3){\vector(0,1){1.4}}
\put(1.4,0.2){\makebox(0,0){$\beta$}}
\put(1.4,2.2){\makebox(0,0){$\beta$}}
\put(-0.5,1.1){\makebox(0,0){$\delta$}}
\put(3.5,1.1){\makebox(0,0){$\delta$}}
\end{picture}
\end{center}

\end{proposition}

\begin{proposition}\label{gstandard} Under the convention that the apotomic projection $f$ is a standard morphism, the apo-syntonic projection $g$ also turns out to be a special standard morphism.
\end{proposition}

\begin{proof}
As $f$ is special standard we have a decomposition $w = f(ac) = f(a) f(c) = tca sac$ with a negative (= plagal) standard word  $tca$ and a positive (= authentic) standard word  $sac$. We have $g(bc) = u_1u_2 \dots u_{|w|} \in \{b, c\}^\ast$ with $$u_k := \left \{ 
\begin{array}{lll} 
b & \mbox{if} & w_k = a, \\ 
b & \mbox{if} & w_k = c \mbox{ and } |w_1 \dots w_k|_c \mbox{ is odd} \\
c & \mbox{if} & w_k = c \mbox{ and }  |w_1 \dots w_k|_c  \mbox{ is even.} 
\end{array} \right.$$                                                                                                          
The final letter of $u$ is $c$, because $|w|_c$ is even, so by definition of $u$ above, the last letter $c$ is fixed. Let $v = cuc^{-1} \in \{b, c\}^\ast$ denote the result of conjugating $u$ with $c^{-1}$.  In order to show that $g$ is a special standard morphism, it is sufficient to show that $v = v_1v_2 \dots v_{|w|}$ is the bad conjugate of $u$. Let $m = |g(b)|$ denote the length of $g(b)$. It is sufficient to show that $|v_1 \dots v_m|_b$ differs from $|u_1 \dots u_m|_b = |g(b)|_b.$  Knowing that $f(a)$ is a prefix of $f(c)$ we write $w = f(a) f(a) f(a^{-1}c) = tca tca rac$ and we may conclude that $w_m = a$, and hence $u_m = b = v_{m+1}$. But in the light of $v_1 = c$ this implies $|v_1 \dots v_m|_b = |g(b)|_b - 1$, i.e. $v$ is the bad conjugate.
\end{proof}

\section{PWWF Substitutions and Automorphisms of $F_3$}\label{Family}

In the two-dimensional situation of well-formed modes we may interpret the Sturmian morphisms as positive automorphisms of the free group $F_2$. In the world of substitutions on words in three letters their analogues constitute different transformational concepts. The PWWF substitutions are well-adopted to the family of (non-singular) pairwise well-formed modes. Still there is a small subfamily of pairwise well-formed modes, where these substitutions are also automorphisms of the free group $F_3$. This final section is dedicated to their study.                                                                                                                                                                                                                                                  
To get a concrete idea about the role of $F_3$-automorphisms, we look into the Authentic and Triadic Divisions of the Phrygian mode (see Figure \ref{fig:Phrygian}). In the theory of well-formed modes one describes the species $baaa$ and $baa$ of the fifth and fourth as images of $a$ and $b$ under a word transformation $\tilde{g}(a) = baaa$, $\tilde{g}(b) = baa$. The left side of Figure \ref{fig:Phrygian} shows a decomposition of this compound transformation into three elementary ones. First, the fifth $a$ is filled with a fourth and a major step $a \mapsto ba$, then both fourths are filled with a minor third and a major step $b \mapsto ba$ and finally, both minor thirds are filled with a minor step and a major step $b \mapsto ba$.     

\begin{figure}
\begin{center}
\begin{minipage}{80 mm}
\resizebox*{8 cm}{!}
{\includegraphics{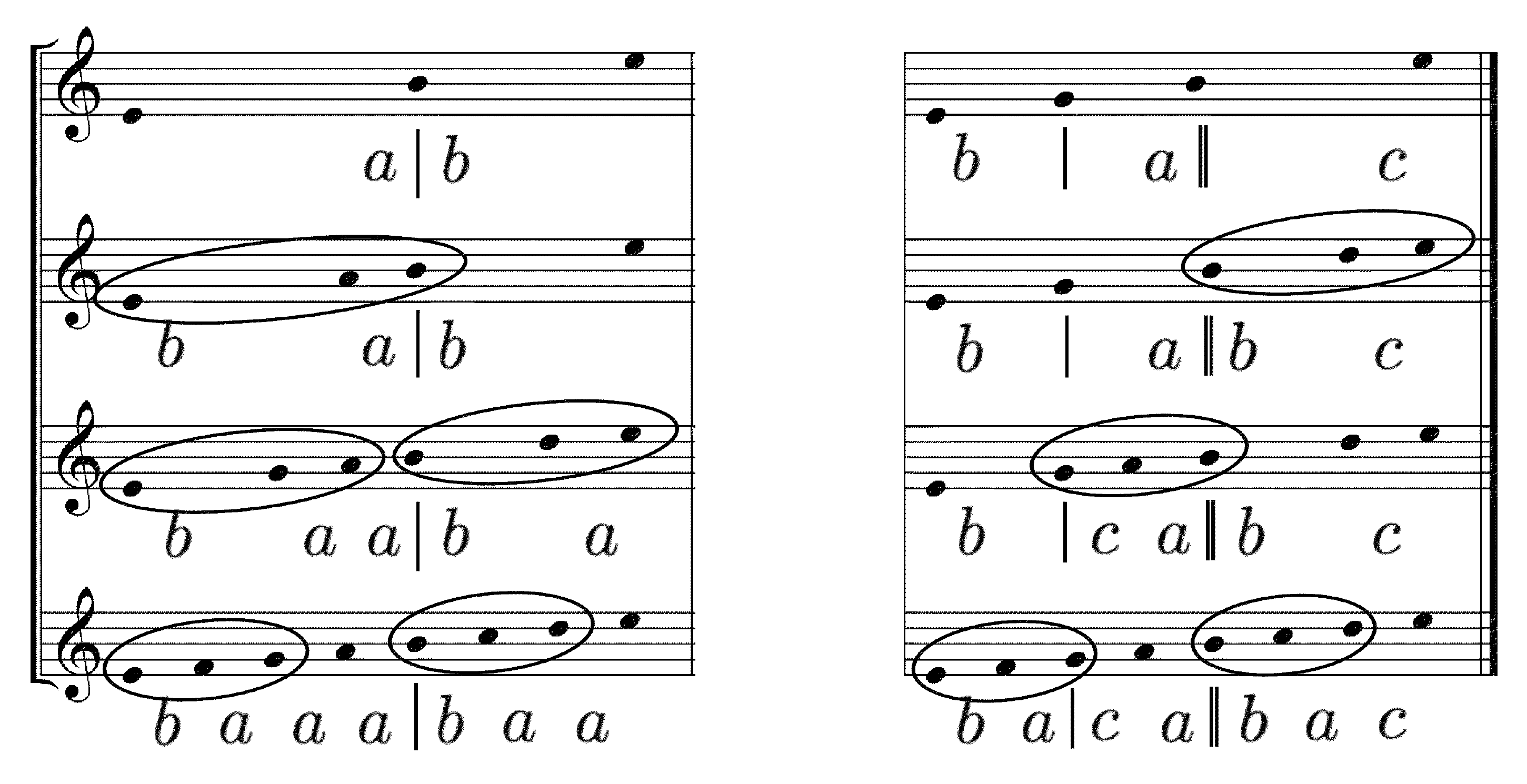}}%
\caption{Construction of the Authentic Phrygian and the Triadic Phrygian Minor modes through substitutions.}%
\label{fig:Phrygian}
\end{minipage}
\end{center}
\end{figure}

The right side of Figure \ref{fig:Phrygian} shows an analogous procedure for the construction of the Phrygian Minor mode $ba|ca||bac$. The fourth is filled with a minor third and a lesser major step: $c \mapsto bc$. Then the major third is filled with a lesser and and greater major step: $a \mapsto ca$ and finally both minor thirds are filled with a minor step followed by a greater major step: $b \mapsto ba$. This final act in the generation of the Phrygian Minor mode does not work analogously for the Ionian Major mode, because there we find two different species of the minor third: $ba$ and $ab$. 

The automorphism group $Aut(F_n)$ of the free group $F_n = \left<x_1, x_2, \dots, x_n \right>$ is redundantly generated by the elementary \emph{Nielsen transformations} (see \cite{Nielsen24}, \cite{CombinatorialGroupTheory}, p. 162 ff.), namely the letter transpositions $x_i \mapsto x_k, x_k \mapsto x_i$, cyclic letter permutations $x_1 \mapsto x_2 \mapsto \dots \mapsto x_n \mapsto x_1$, letter inversions $x_i \mapsto x_i^{-1}$ and the substitutions of the types $x_i \mapsto x_i x_k$ or  $x_i \mapsto x_k x_i$. The automorphisms of $F_3$, which potentially coincide with PWWF substitutions, are necessarily positive, so we don't need to consider letter inversions here. This leads to the following definition: 

\begin{definition} An authentic PWWF substitution is called morphic, if it is a positive automorphism of $F_3$.  \end{definition}
For two different letters $x, y \in \{a, b, c\}$ let $E_{x,y}, A_{xy}, P_{xy} \in Aut(F_3)$ denote the following positive automorphisms of the free group $F_3$: Let $z \in \{a, b, c\}$ denote the third letter, respectively.
$$\begin{array}{lll}
E_{xy}(x) = y, & E_{xy}(y) = x, & E_{xy}(z) = z,\\
A_{xy}(x) = xy, & A_{xy}(y) = y, & A_{xy}(z) = z,\\
P_{xy}(x) = yx, & P_{xy}(y) = y, & P_{xy}(z) = z.\\
\end{array}$$  Hence, a PWWF substitution $\sigma: \{a, b, c\}^\ast \to  \{a, b, c\}^\ast$ is morphic, if it can be written as a composition of a finite number of letter permutations $E_{xy}$ and substitutions of the types  $A_{xy}$ and/or $P_{xy}$. 
Now we inspect the seven conjugates of the Phrygian minor mode to which also the Ionian major mode belongs. The following proposition is thus the portrait of a very special conjugation class of modes. As in the two-letter case it contains bad conjugates.

\begin{proposition} Let $\sigma: \{a, b, c\}^\ast \to \{a, b, c\}^\ast$ with $\sigma(b) = ba$, $\sigma(a) = ca$ and $\sigma(c) = bac$ denote the PWWF substitution, which is associated with the Phrygian minor mode. The cycle of the single letter conjugations $ba|ca||bac \to ac|ab||acb \to                                                                                                                                                                                                                                                                                                                                                             ca|ba||cba  \to  ab|ac||bac  \to ac|ba||cab \to cb|ac||aba \to  ba|ca||bac$ contains two bad conjugates and --- accordingly --- five PWWF modes. Four of these PWWF modes are morphic. The exceptional good, but amorphous instance, is the Ionian major mode. These seven conjugates together with their associated projections under $\pi_{b \to c}$, $\pi_{c \to a} (\sigma)$ and  $\pi_{a \to b} (\sigma)$ are listed below: 
$$\begin{array}{|l||l|l|l||l|}
 \hline
\quad \mbox{authentic} & \quad \mbox{apotomic} & \quad \mbox{syntonic} \quad & \quad \mbox{apo-syntonic}  & \quad \mbox{transform.} \\
\quad \mbox{triadic mode} \quad & \quad \mbox{projection}  & \quad \mbox{projection}  & \quad \mbox{projection}  & \quad \mbox{type} 
\\ \hline
\quad ba|ca||bac \quad & \quad ca|cacac \quad & \quad baaa|baa \quad & \quad bbcb|bbc \quad & \quad \mbox{morphic}^{\ast \ast} \\
\quad ac|ab||acb & \quad ac|acacc & \quad aaab|aab & \quad bcbb|bcb &  \quad \mbox{morphic} \\
\quad ca|ba||cba & \quad ca|cacca & \quad aaba|aba & \quad cbbb|cbb & \quad \mbox{morphic} \\
\quad ab|ac||bac & \quad ac|accac &  \quad abaa|baa & \quad bbbc|bbc & \quad \mbox{morphic} \\ \hline
\quad ba|cb||aca & \quad ca|ccaca & \quad baab|aaa  & \quad bbcb|bcb & \quad \mbox{bad}^\ast \\
\quad ac|ba||cab & \quad ac|cacac &\quad aaba|aab & \quad bcbb|cbb & \quad \mbox{good}^\ast \\
\quad cb|ac||aba & \quad cc|acaca & \quad abaa|aba & \quad cbbc|bbb & \quad \mbox{bad}^{\ast \ast} \\ \hline
\end{array}$$\end{proposition}
\begin{proof}
In the left column of the table $ba|ca||bac $ undergoes the full cycle of letter-by-letter conjugations. In parallel the projections $\pi_{b \to c} (\sigma)$, $\pi_{c \to a} (\sigma)$ , $\pi_{a \to b} (\sigma)$ run through their corresponding conjugations. By virtue of proposition \ref{gstandard} the conjugations of the apotomic and the apo-syntonic projections are ``in sync", i.e. they start with special standard modes (marked as $\mbox{morphic}^{\ast \ast}$) and they end with bad modes (marked as $\mbox{bad}^{\ast \ast}$). As a consequence there are only two bad conjugates among the triadic modes. The generation of the four morphic modes --up to letter permutations -- is given below: 
$$\begin{array}{lllllll}
A_{ba} P_{ac}  P_{cb} (b|a||c) & = & A_{ba} P_{ac} (b|a||bc) & = & A_{ba}(b|ca||bc) & = & ba|ca||bac \\
P_{ca} A_{ab} P_{bc} (c|a||b) & = & P_{ca} A_{ab} (c|a||cb) & = & P_{ca}(c|ab||cb) & = & ac|ab||acb \\
A_{ba} P_{ac} A_{cb} (a|b||c) & = & A_{ba} P_{ac} (a|b||cb)  & = & A_{ba}(ca|b||cb) & = & ca|ba||cba \\
P_{ca} A_{ab} P_{cb} (b|a||c) & = & P_{ca} A_{ab} (b|a||bc) & = & P_{ca}(ab|c||bc) & = & ab|ac||bac \\
\end{array}$$ The obstacle for the ``good" PWWF Ionian major mode to be generated by an $F_3$-automorphism is the co-existence of the two factors $ab$ and $ba$. 
 \end{proof}
 
 In addition to the $\mbox{bad}^\ast$ Locrian mode (with a bad syntonic projection) there is the $\mbox{bad}^{\ast \ast}$ Dorian minor mode (with bad apotomic and  apo-syntonic projections), whose structural defects have been discussed in 19th-century treatises, such as Moritz Hauptmann's 1853 \emph{Die Natur der Harmonik und der Metrik}.  

We show now that a similar picture arises in connection with a certain family $\mathcal{M}$ of morphic PWWF modes and we conclude the article with the conjecture that this family actually exhausts the morphic PWWF modes entirely.
It is useful to start the investigation of this family from the authentic standard modes, whose bisecting transformations then yield the associated PWWF modes. The general form of the standard morphism $f$, generating the single-divider mode $f(a)|f(c)$  is $f = G^kDG^{2n}$ with $n > 0$ and $k \ge 0$. The corresponding  apo-syntonic conversion of $f$, the standard morphism $g$, generating the first of the two double divider modes $g(b)|g(c)$  is $g = G^{2k+2}DG^{n-1}$.
Thus, we have the two sets $$\mathcal{F} = \{G^kDG^{2n} \, | \, n  > 0, k  \ge 0 \} \mbox{   and   } \mathcal{G} = \{G^{2n}DG^{k} \, | \, n  > 0, k  \ge 0 \}$$ together with the \emph{apo-syntonic conversion} $\theta: \mathcal{F} \to \mathcal{G}$, where $\theta(G^kDG^{2n}) := G^{2k+2}DG^{n-1}$.
Furthermore, we consider reversal map, i.e. the unique anti-auto\-morphism of $rev: \left<G, D \right> \to \left<G, D \right>$ fixing both $G$ and $D$. This map $rev$ sends $\mathcal{F}$ to $\mathcal{G}$ and vice versa.

The following proposition specifies the commutative diagram for matrices in proposition \ref{commdiag}  to the elements of the sets $\mathcal{F}$ and $\mathcal{G}$:
\begin{proposition} $\theta \circ rev \circ \theta = rev.$
\end{proposition}  
\begin{proof} Although the relation is a corollary of proposition \ref{commdiag} we give a direct proof here: $$\begin{array}{lll} \theta(rev(\theta(G^kDG^{2n}) & = & \theta(rev(G^{2k+2}DG^{n-1}))\\ & = & \theta(G^{n-1}DG^{2k+2}) \\ 
& = & G^{2n}DG^{2k+2} = rev(G^kDG^{2n}) \end{array}$$
\end{proof}

The third corresponding special Sturmian morphism $\tilde{g}$, generating the syntonic projection $\tilde{g}(a|b)$ of our PWWF mode  has the following form: $$\tilde{g} = \tilde{\theta}(f) = \tilde{\theta}(G^kDG^{2n}) = G^k \tilde{G}^{k+2}DG^{n-1}.$$ 
Thus, the morpism $\tilde{g}$ is preceded by precisely $k+2$ conjugate morphisms in the Zarlino ordering, namely by $G^{k+l} \tilde{G}^{k+2-l}DG^{n-1}$ for $l = 1, \dots, k+2$.  For the corresponding incidence matrices we have:
$$\begin{array}{ll}M_f = \m {k+1} {2n(k+1)+k} {1} {2n + 1} \quad & \mbox{and} \quad
 M_g = M_{\tilde{g}} = \m {2k+3} { n(2k+3) -1} {1} {n} \end{array}$$ 

\begin{proposition}\label{morphicPWWF} Consider the authentic PWWF three-letter mode $$v_{k, n} := a^kba|a^kca||(a^kba \, a^kca)^{n-1}a^kba \, a^kc$$ Its apotomic, apo-syntonic and syntonic projections are $\pi_{b \to c} (v_{k, n}) = f(a|c)$, $\pi_{a \to b} (v_{k, n}) = g(b|c)$ and $\pi_{c \to a} (v_{k, n}) = \tilde{g}(a|b)$, respectively.
\end{proposition}  
\begin{proof}
$$\begin{array}{lll}
f(a|c) & = & G^kDG^{2n}(a|c) = G^kD(a|a^{2n}c) = G^k(ca|(ca)^{2n}c) \\
& = & a^kca|(a^kca)^{2n}a^kc = \pi_{b \to c} (v_{k, n})\\

g(b|c) & = & G^{2k+2}DG^{n-1}(b|c) = G^{2k+2}D(b|b^{n-1}c) = G^{2k+2}(cb|(cb)^{n-1}c) \\
& = & b^{2k+2}cb|(b^{2k+2}cb)^{n-1}b^{2k+2}c = \pi_{a \to b} (v_{k, n}). \\
\tilde{g}(a|b) & = & G^k\tilde{G}^{k+2}DG^{n-1}(a|b) = G^k\tilde{G}^{k+2}D(a|a^{n-1}b) = G^k\tilde{G}^{k+2}(ba|(ba)^{n-1}b), \\
& = & G^k\tilde{G}^{k+2}(ba|(ba)^{n-1}b) = G^k(ba^{k+3}|(ba^{k+3})^{n-1}ba^{k+2})\\
& = & a^k ba^{k+3}|(a^k ba^{k+3})^{n-1} a^k ba^{k+2} = \pi_{c \to a} (v_{k, n}), \\

\end{array}$$ 
\end{proof}

The subsequent proposition provides an explicit portrait of the full conjugation class of the authentic PWWF mode $v_{k, n}.$
\begin{proposition}\label{morphiclist} The table below lists all letter-by-letter conjugations of the authentic PWWF mode $v_{k, n}$ and characterizes them as morphic, good or bad. The segment between the $\mbox{bad}^{\ast \ast}$ mode (with bad apotomic and apo-syntonic projections) and the $\mbox{bad}^\ast$-mode (with bad syntonic projection) is exclusively occupied by morphic modes. The opposite segment between the $\mbox{bad}^\ast$-mode and the $\mbox{bad}^{\ast \ast}$ mode is exclusively occupied by good modes: 
$$\begin{array}{llllll}
\hline 
a^kba & | & a^kca & || & (a^kba \, a^kca)^{n-1}(a^kba)(a^kc)  & \mbox{morphic}^{\ast \ast}\\
a^{k-1}ba^2 & | & a^{k-1}ca^2 & || & (a^{k-1}ba^2 \, a^{k-1}ca^2)^{n-1}(a^{k-1}ba^2)(a^{k-1}ca) & \mbox{morphic}\\
\dots \\
a^{k-l}ba^{l+1} & | & a^{k-l}ca^{l+1}  & || & (a^{k-l}ba^{l+1}  \, a^{k-l}ca^{l+1} )^{n-1}(a^{k-l}ba^{l+1})(a^{k-l}ca^l) & \mbox{morphic}\\
\dots \\
ba^{k+1} & | & ca^{k+1}  & || & (ba^{k+1}  \, ca^{k+1} )^{n-1}(ba^{k+1})(ca^k) & \mbox{morphic}\\
a^{k+1}c & | & a^{k+1}b  & || & (a^{k+1}c  \, a^{k+1}b)^{n-1}(a^{k+1}c)(a^kb) & \mbox{morphic}\\
\dots \\
ca^{k+1} & | & ba^{k+1}  & || & (ca^{k+1} \, ba^{k+1})^{n-1}(ca^k)(ba^{k+1}) & \mbox{morphic}\\
a^{k+1}b & | & a^{k+1}c  & || & (a^{k+1}b \, a^{k+1}c)^{n-1}(a^kb)(a^{k+1}c) & \mbox{morphic}\\
\dots \\
ba^{k+1} & | & ca^{k+1}  & || & (ba^{k+1} \, ca^{k+1})^{n-2}(ba^{k+1} \, ca^k)(b a^{k+1} c a^{k+1}) & \mbox{morphic}\\
a^{k+1}c & | & a^{k+1}b  & || & (a^{k+1}c \, a^{k+1}b)^{n-2}(a^{k+1}c \, a^kb)(a^{k+1}c a^{k+1}b) & \mbox{morphic}\\
\dots \\
ca^{k+1} & | & ba^{k+1}  & || & (c a^{k+1} \, b a^{k+1})^{n-2}(c a^k \, b a^{k+1})(ca^{k+1} ba^{k+1}) & \mbox{morphic}\\
a^{k+1}b & | & a^{k+1}c  & || & (a^{k+1}b \, a^{k+1}c)^{n-2}(a^kb \, a^{k+1}c)(a^{k+1}b a^{k+1}c) & \mbox{morphic}\\
\dots \\
aba^k & | & a c a^k & || & b a^k a c a^{k} (a b a^k a c a^{k})^{n-1}  & \mbox{morphic}\\
\hline 
ba^{k+1} & | & c a^kb & || & (a^{k+1} c a^{k+1}b)^{n-1} a^{k+1} c a^{k+1} & \mbox{bad}^{\ast}\\
\hline 
a^{k+1}c & | & a^kba & || & (a^k c a^{k+1}ba)^{n-1} a^k c a^{k+1}b & \mbox{good}^{\ast}\\
\dots \\
aca^k & | & ba^{k+1} & || & (c a^{k+1}ba^{k+1})^{n-1}c a^{k+1}ba^k & \mbox{good} \\
\hline 
ca^kb & | & a^{k+1}c & || & (a^{k+1}ba^{k+1}c)^{n-1}a^{k+1}ba^{k+1} & \mbox{bad}^{\ast \ast} \\
\hline 
\end{array}$$  
\end{proposition}

\begin{proof}
The following calculation shows that $v_{k, n}$ is morphic. The calculations for the other morphic modes are analogous:

$$\begin{array}{lll}
P_{ba}^k P_{ca}^k A_{ba} P_{ac} (P_{cb} P_{ca})^{n-1} P_{cb} E_{ab} (a|b||c) 
& = & P_{ba}^k P_{ca}^k A_{ba} P_{ac} (P_{cb} P_{ca})^{n-1} P_{cb} (b|a||c) \\
& = & P_{ba}^k P_{ca}^k A_{ba} P_{ac} (P_{cb} P_{ca})^{n-1} (b|a||bc) \\
& = & P_{ba}^k P_{ca}^k A_{ba} P_{ac} (b|a||(ba)^{n-1}bc) \\
& = & P_{ba}^k P_{ca}^k A_{ba} P_{ac} (b|a||(ba)^{n-1}bc) \\
& = & P_{ba}^k P_{ca}^k A_{ba} (b|ca||(bca)^{n-1}bc) \\
& = & P_{ba}^k P_{ca}^k(ba|ca||(ba \, ca)^{n-1}ba \, c) \\
& = & P_{ba}^k(ba|a^kca||(ba \, a^kca)^{n-1}ba \, a^kc) \\
& = & a^kba|a^kca||(a^kba \, a^kca)^{n-1}a^kba \, a^kc \\
& = & v_{k, n}
\end{array}$$

We try to write the good$^{\ast\ast}$-mode $\gamma_{k, n} = a^{k+1}c | a^kba || (a^k c a^{k+1}ba)^{n-1} a^k c a^{k+1}b$ as an image $f(a|b||c)$ under an automorphism $f = f_m f_{m-1} \dots f_{1} f_{0}$, where the $f_i (i = 1, .., m)$ are supposed to be productions of the type $A_{xy}$ or $P_{xy}$ and where $f_0$ is letter permutation. First we observe that $f_m$, the last of these morphisms, cannot be a production of $b$'s or $c$'s: First of all, $A_{bc}$, $A_{cb}$, $P_{bc}$ and $P_{bc}$ are excluded because there are no letters $c$ and $b$ neighboring each other (even in the case $k = 0$). But furthermore, not all instances of the letter $a$ are followed or preceded by either exclusively $b$ or exclusively $c$, and so also the productions $A_{ac}$, $A_{ab}$, $P_{ac}$ and $P_{ac}$ can be excluded excluded. The only remaining possibilities are productions of the letter $a$. Among these the append-transformations $A_{ba}$ and  $A_{ca}$ both excluded, as the single-dividiver prefix $a^{k+1}c$ of $\gamma_{k, n}$ ends on $c$ and the double-divider suffix $(a^k c a^{k+1}ba)^{n-1} a^k c a^{k+1}b$ ends on $b$. For $k > 0$ the prepend-transformations $P_{ba}$ and  $P_{ca}$ are suitable in order to produce $\gamma_{k, n}$ from shorter words. To be more precise: $P_{ba}$ and  $P_{ca}$ commute with each other and both can be applied $k$ times in any order to the triple $\gamma_{0, n} = ac | ba || (c aba)^{n-1} cab$ to produce $\gamma_{k, n}$. Here is one of them: 

$$\begin{array}{lll}
a^{k+1}c | a^kba || (a^k c a^{k+1}ba)^{n-1} a^k c a^{k+1}b
& = & P_{ba}^k(a^{k+1}c | ba || (a^k c aba)^{n-1} a^k c ab)\\                                                                                                                                                                                                                                                                                                                                                                                                                                            
& = & P_{ca}^k(P_{ba}^k(ac | ba || (c aba)^{n-1} cab))\\
\end{array}$$
                                                                                                                       
A closer look at $\gamma_{0, n}$ shows that it is not an image of a shorter word-triple under any of the 8 transformations of the type $A_{xy}$ or $P_{xy}$.  

The similar line of argument works for any of the good modes. For $0 < l \le k$ the general form of a good mode is $$\begin{array}{ll} & a a^{k-l}ca^l | a^{k-l}ba^{l}a || (a^{k-l} c a^{l} a a^{k-l} b a^{l} a)^{n-1} a^{k-l} c a^{l} a a^{k-l} b a^{l} \\ = & P_{ca}^{k-l}(P_{ba}^{k-l}(A_{ca}^l(A_{ba}^l(\gamma_{0, n}))))\end{array}$$      

So all the good modes are images of the mode $\gamma_{0, n}$. And this is the only possibility to generate them from a shorter triple. \end{proof}

\begin{conjecture} Consider an authentic PWWF substitution $\sigma$ in the sense of definition \ref{PWWFsubstitution}. The substitution $\sigma$ is morphic (i.e. is a positive automorphism of $F_3$) iff its generated mode $\sigma(a) | \sigma(b) || \sigma(c)$ --- up to letter permutation --- is an instance of a morphic mode in proposition \ref{morphiclist}.
\end{conjecture}

\section{Conclusion} 
Since the rise of the triad in its role as a governing concept in the music of harmonic tonality, music theorists have been in a quandary as to how to adjust the immemorial diatonic scale so as to be compatible with the triad's new primacy. The interval of the major third challenges the status of the perfect fifth as a scale generator, and in the course of this competition it undermines the validity of other properties that are consequences of the fifth-generatedness of the diatonic scale, such as the well-formedness property. The concept of pairwise well-formedness offers a reconciliation, insofar as it implements the idea of a coexistence of three well-formed scale structures within one parent scale. The letter projections mediate between the competing interpretations of a fifth- and a third-generated scale. The present paper offers a transformational upgrade to that earlier basic insight. Following the pattern of the investigation of well-formed modes through automorphisms of the free group $F_2$ it clarifies the combinatorial behavior of all non-singular pairwise well-formed modes. For the concrete case of the Ionian major mode $ac|ba||cab$, it turns out that the underlying substitution is not an automorphism of the free group $F_3$, which is an exception in comparison to the Phrygian minor, Lydian major, Mixolydian major and Aeolian minor modes. This exceptional status corresponds to the musical fact that both, the species of major third $ac$ and $ca$ as well as those of the minor third $ba$ and $ab$ are different. The mathematical status of the Dorian minor mode as a bad mode has its musical counterpart in the fact that the species $cbac$ doesn't form a proper fifth. This fact in turn is reflected in the special treatment some nineteenth-century theorists accorded the supertonic ii harmony in major, as a sort of cousin to the diminished supertonic triad in minor.

\end{document}